\documentclass[letterpaper,12pt,oneside]{amsart}
\usepackage{amsmath}
\usepackage{amssymb}
\usepackage{latexsym}
\usepackage{amsfonts}
\usepackage[latin1]{inputenc}
\usepackage[english]{babel}
\usepackage{multicol}
\usepackage{xcolor}
\usepackage{oldgerm}
\usepackage{enumerate}
\usepackage{hyperref}
\usepackage{lipsum}
\usepackage{amsfonts}
\usepackage{mathtools}
\usepackage{multicol}
\newcommand{\seqnum}[1]{\href{https://oeis.org/#1}{\underline{#1}}}

\theoremstyle{plain}
\newtheorem{theorem}{Theorem}
\newtheorem{corollary}[theorem]{Corollary}
\newtheorem{lemma}[theorem]{Lemma}

\theoremstyle{definition}
\newtheorem{definition}[theorem]{Definition}
\newtheorem{example}[theorem]{Example}
\newtheorem{conjecture}[theorem]{Conjecture}

\theoremstyle{remark}
\newtheorem{remark}[theorem]{Remark}

\newcommand{\Z}{{\mathbb Z}}

\newcommand{\A}{{\mathcal A}}

\def\A{\alpha}

\def\B{\beta}

\def\L{\Lambda}

\def\G{\gamma}

\def\z{\zeta}

\def\Z{\mathbb{Z}}
\def\K{\mathbb{K}}
\def\Q{\mathbb{Q}}

\def\el{\ell}

\begin{document}

\title[Repdigits and the Narayana sequence]{Repdigits in Narayana's Cows Sequence and their Consequences}

\author{Jhon J. Bravo}
\address{Departamento de Matem\'aticas\\ Universidad del Cauca\\ Calle 5 No 4--70\\Popay\'an, Colombia.}
\email{jbravo@unicauca.edu.co}

\author{Pranabesh Das}
\address{University of Waterloo \\ Department of Pure Mathematics \\ Waterloo \\ Canada}
\email{pranabesh.math@gmail.com}

\author{Sergio Guzm\'an}
\address{Facultad de Ciencias en F\'isica y Matem\'aticas\\ Universidad Aut\'onoma de Chiapas\\ Mexico}
\email{strebeinsam@gmail.com}



\begin{abstract}
Narayana's cows sequence satisfies the third-order linear recurrence relation $N_n=N_{n-1}+N_{n-3}$ for $n \geq 3$ with initial conditions $N_0=0$ and $N_1=N_2=1$. In this paper, we study $b$-repdigits which are sums of two Narayana numbers. We explicitly determine these numbers for the bases $2\le b\leq100$ as an illustration. We also obtain results on the existence of Mersenne prime numbers, 10-repdigits, and numbers with distinct blocks of digits in the Narayana sequence. The proof of our main theorem uses lower bounds for linear forms in logarithms and a version of the Baker-Davenport reduction method in Diophantine approximation. 

\medskip

\noindent\textbf{Keywords and phrases.}\, Narayana sequence, $b$-repdigit, linear forms in logarithm, reduction method.

\noindent\textbf{2010 Mathematics Subject Classification.}\, 11B83, 11J86.

\end{abstract}

\maketitle

In 1356, the Indian mathematician Narayana Pandit wrote his famous book titled \emph{Ganita Kaumudi} where he proposed the following problem of a herd of cows and calves: \emph{A cow produces one calf every year. Beginning in its fourth year, each calf produces one calf at the beginning of each year. How many calves are there altogether after 20 years?} \cite{Allouche96}.

We can translate this problem into our modern language of recurrence sequences. We observe that the number of cows increased by one after one year, increased by one after two years, increased by one after three years and increased by two after four years and so on. Hence we obtain the sequence $1,1,1,2,\ldots$. In the $n$-th year, Narayana's problem can be written as the following linear recurrence sequence:
\[
N_n=N_{n-1} + N_{n-3}
\]
for $n\geq 3$ with $N_0=0$, $N_1=N_2=1$ as initial conditions. The first few terms of the sequence are 
\[
0,1,1,1,2,3,4,6,9,13,19,\ldots \quad \text{(sequence \seqnum{A000930})}.
\]
In this sequence each number is computed recursively by adding the previous number in the sequence and the number two places previous to the number. The defining relation in the Narayana sequence is very similar to the famous Fibonacci sequence but with a delay in the recursion which makes it a third-order linear recurrence sequence. This can be thought of as a ``delayed morphism" and has interesting applications in automata theory. It has been considered by Allouche and Johnson \cite{Allouche96}.

Let $b\geq 2$ be an integer. A positive integer greater than $b$ is said to be a repdigit in base $b$, or simply a $b$-repdigit, if it has only one distinct digit in its base $b$ representation. In particular, such numbers have the form $a(b^{\ell}-1)/(b-1)$ for some $\ell \geq 2$ and $1\leq a\leq b-1$. For example, $11$ is a repdigit in base $10$ whereas $399$ is not a repdigit in base $10$. Although $399= 19\cdot 20+ 19=[19,19]_{20}$ shows that it is a repdigit in base $20$. We omit to mention the base and simply write repdigit when the base $b$ is $10.$

There are several papers in the literature that have considered diophantine equations involving repdigits in the Fibonacci, Lucas or Pell sequences. For example, Luca \cite{FL} showed that 55 and 11 are the only repdigits in the Fibonacci and Lucas sequences, respectively. Faye and Luca \cite{Bernadette-Luca1} proved that there are no repdigits in the Pell sequence. Luca, Normenyo, and Togb\'e in \cite{NormenyoLucaTogbe3,NormenyoLucaTogbe2} determined repdigits which are sums of four Fibonacci, Lucas or Pell numbers.  The special cases of repdigits expressible as sums of three Fibonacci, Lucas or Pell numbers were solved earlier by Normenyo, Luca, and Togb\'e \cite{LuRep,NormenyoLucaTogbe4,NormenyoLucaTogbe1}.

In this paper, we are interested in finding all $b$-repdigits which are the sum of two Narayana numbers for the bases $2\le b\leq 100$. More precisely, we determine all the solutions of the Diophantine equation
\begin{equation}\label{eqn}
N_n + N_m = [a,\ldots, a]_b =a \left(\frac{b^\ell-1}{b-1}\right),
\end{equation}
in integers $(n,m,\ell,a,b)$ with $0\leq m\leq n$, $2\le b\leq 100$, $1\leq a \leq b-1$ and  $\ell\geq2$. 

Several authors have investigated variants of this problem. For instance, Bollman, Hern\'andez, and Luca \cite{BHL} found all Fibonacci numbers which are sums of three factorials. Luca and Siksek \cite{LuSi} found all  factorials that can be written as sums of two and three Fibonacci numbers. Furthermore, D\'iaz and Luca \cite{D-L} determined all Fibonacci numbers that are the sum of two repdigits. Bravo, Luca and several other authors considered \cite{BGL-prep-14,BGL2,BL1,BLMonatsh15,DM} similar problems in generalized Fibonacci numbers.  

Before presenting our result, we note that if $\ell=2$, then equation \eqref{eqn} can be written as $N_n + N_m= ab+a=[a,a]_b$. We have computed the long list of trivial solutions but we will not list them in this paper. The solutions to the equation \eqref{eqn} are not particularly interesting for $\ell=2$. In this paper, we call them trivial solutions. We are particularly interested in solutions (non-trivial) for $\ell \ge 3$. 

Our main result on the solutions of equation \eqref{eqn} is the following: 

\begin{theorem}\label{mainthm}
The Diophantine equation
\begin{equation*}
N_n + N_m = a \left(\frac{b^\ell-1}{b-1}\right),
\end{equation*}
has only finitely many non-trivial solutions in integers $n,m,\ell,a,b$ with $0\leq m\leq n$, $2\le b\leq 100$, $1\leq a \leq b-1$ and  $\ell\geq3$. Moreover, all the solutions for  $\ell\geq 3$ are denoted by the tuple $(n,m,\ell,a,b)$ and they are listed in Section \ref{sec5}. In particular, equation \eqref{eqn} has no solutions for $\ell \geq 7$. 
\end{theorem}

As a consequence we obtain the following corollaries.

\begin{corollary}\label{coro}
All the solutions of the Diophantine equation 
\[
N_n = a\left(\frac{b^\ell-1}{b-1}\right),
\]
in non-negative integers $n,\ell,a,b$ with $2\le b\leq100$, $1\leq a \leq b-1$ and  $\ell\geq 3$ are given by the tuples
\[
(n,\ell,a,b)\in\{(9,3,1,3),(15,3,3,6)\}.
\] 
Namely, we have $N_9=13=[1,1,1]_3$ and $N_{15}=129=[3,3,3]_6$.  
\end{corollary}

\begin{corollary}\label{coro2}
The only repdigit in the Narayana sequence is $N_{14}=88$. In addition, there are no Mersenne prime numbers in the Narayana sequence. 
\end{corollary}

Other consequences related to Narayana numbers which have only one distinct block of digits are discussed in section \ref{sec5} of this document (p.\ \pageref{pag}).

We would like to note that the recurrence relation for the Narayana sequence might look similar to the Fibonacci sequence but they are very different. In fact, the Narayana sequence is a recurrence sequence of order three, hence we do not get many nice properties of binary recurrence sequences. The main tools used in this paper to prove the main result are lower bounds for linear forms in logarithms of algebraic numbers and a version of the reduction procedure due to Baker and Davenport \cite{BD}. 

\section{Preliminaries and notations}
We begin this section by giving a formal definition of the Narayana sequence and some of its properties.

\begin{definition}
The Narayana sequence $(N_n)_{n\geq 0}$ is defined by the third-order linear recurrence relation $N_n=N_{n-1} + N_{n-3}$ for $n\geq 3$, where the initial conditions are given by $N_0=0$ and $N_1=N_2=1$.
\end{definition}
We next mention some facts about the Narayana sequence. First, it is known that the characteristic polynomial for $(N_n)_{n\geq 0}$ is given by
\[
f(x)= x^3 - x^2 -1.
\]
This polynomial is irreducible in $\mathbb{Q}[x]$. We note that it has a real zero $\alpha$ ($>1$) and two conjugate complex zeros  $\beta$ and $\gamma$ with $|\beta|=|\gamma|<1$. In fact, $\alpha \approx 1.46557$. We also have the following properties of $(N_n)_{n\geq 0}$.  

\begin{lemma}\label{lem:Prop}
For the sequence $(N_n)_{n\geq 0}$, we have 

\begin{enumerate}
\item[$(a)$] $\alpha^{n-2}\leq N_n \leq \alpha^{n-1}$ for all $n\geq 1$.
\item[$(b)$] $(N_n)_{n\geq 0}$ satisfies the following ``Binet-like" formula 
\[
N_n= a_1 \alpha^{n} + a_2 \beta^{n}+ a_3 \gamma^{n} \quad \text{for all} \quad n\geq 0, \quad \text{where}
\]
\[
a_1=\frac{\alpha}{(\alpha-\beta)(\alpha-\gamma)}, \quad  a_2=\frac{\beta}{(\beta-\alpha)(\beta-\gamma)} \quad \text{and} \quad a_3=\frac{\gamma}{(\gamma- \alpha)(\gamma-\beta)}.
\]
\item[$(c)$] The above Binet-like formula can also be written as 
\[
N_n= C_{\A}\A^{n+2} + C_{\B}\B^{n+2} + C_{\G}\G^{n+2}  \quad \text{for all} \quad n\geq 0, \quad \text{where}
\]
\[
C_x= \frac{1}{x^3+2}.
\]
\item[$(d)$] $1.45<\A<1.5$ and $5<{C_{\A}}^{-1}<5.15$.
\item[$(e)$] If we denote $\z_n= C_{\B}\B^{n+2} + C_{\G}\G^{n+2}$, then $|\z_n|<1/2$ for all $n\geq 1$.
\end{enumerate}
\end{lemma}

\begin{proof}
The first part $(a)$ is a simple exercise in induction on $n$. The proof of $(b)$ can be found in reference \cite{Ramirez2015}, and it is an easy exercise to deduce $(c)$ from part $(b)$. For the proof of $(d)$, we simply note that ${C_{\A}}^{-1}=5.1479\ldots$. Finally, the proof of $(e)$ follows from the triangle inequality and the fact that $|\beta|=|\gamma|<1$. We leave the detail to the reader.  
\end{proof}

\section{Upper bounds for the number of solutions}
We assume throughout that the tuple $(n,m,\ell,a,b)$ represents a solution of equation \eqref{eqn} where $m,n,\ell,a,b$ are positive integers. In the next lemma, we find a relation between $\ell$ and $n$, which we will be used in the proof of the main Theorem \ref{mainthm}.

\begin{lemma}\label{bound_l}
Let $n\geq 4$ and assume that equation \eqref{eqn} holds. Then
\[
(n-2)\frac{\log \A}{\log b}<\ell <n.
\]
\end{lemma}

\begin{proof}
From Lemma \ref{lem:Prop} $(a)$, we get $\A^{n-2}\leq N_n + N_m = a(b^{\ell}-1)/(b-1)< b^{\ell}$. Thus we have 
\[
(n-2)\frac{\log \A}{\log b}<\ell.
\]
Similarly, $b^{\ell-1}<a(b^{\ell}-1)/(b-1)=N_n + N_m\leq 2\A^{n-1}$. By using this and taking into account that $b\geq 2$ and $\A <1.5$, we obtain
\[
\ell < 1+\frac{\log 2}{\log b}+(n-1)\frac{\log \A}{\log b}<2+(n-1)\frac{\log 1.5}{\log 2}< n, 
\]
which holds for all $n\geq 4$.
\end{proof}

We now need to find an upper bound on $n$.

\subsection{An upper bound on n}

Using Lemma \ref{lem:Prop} and equation \eqref{eqn}, we obtain that
\[
C_\A \A^{n+2} - \frac{ab^{\ell}}{b-1}= - N_m - \frac{a}{b-1}- \z_n.
\]
Taking absolute values in the above equality and dividing both sides of the resulting expre\-ssion by $C_\A \A^{n+2}$, we get 
\begin{align*}
\left| 1 - \frac{ab^{\ell}}{C_\A \A^{n+2}(b-1)} \right| & < \frac{N_m}{C_\A \A^{n+2}}+ \frac{3}{2C_\A \A^{n+2}} \leq \frac{\A^{m-1}}{C_\A \A^{n+2}}+ \frac{3}{2C_\A \A^{n+2}} \\
& < \frac{5.15\A^{m-1}}{\A^{n+2}} +\frac{7.725}{\A^{n+2}}.
\end{align*}
Since $1.45 <\A$, we obtain $5.15\A^{m-1}+ 7.725< 17\A^{m-1}$ for all $m\geq 0$, and so
\begin{equation}\label{First_Lin_Form}
 \left| \A^{-(n+2)}b^{\ell}\frac{a}{C_\A (b-1)} - 1 \right| <  \frac{6}{\A^{n-m}}.
\end{equation}
We put
\begin{equation}\label{param-1}
\begin{split}
\gamma_1 &:= \A, \qquad \gamma_2:=b, \qquad \gamma_3:=\frac{a}{C_{\A}(b-1)},\\
& \qquad b_1:=-(n+2), \qquad b_2:=\ell, \qquad b_3:=1,\\
& \qquad\qquad\Lambda_1 := \gamma_1^{b_1}\cdot\gamma_2^{b_2}\cdot\gamma_3^{b_3}-1.
\end{split}
\end{equation}
So we obtain from \eqref{First_Lin_Form} that
\begin{equation}
\label{desl_1}
|\Lambda_1|< \frac{6}{\A^{n-m}}.
\end{equation}
Our next step will be to find a lower bound for $|\Lambda_1|$. For this purpose, we use the following result of Matveev \cite{Matveev} (see also the paper of Bugeaud, Mignotte, and Siksek \cite[Theorem 9.4]{Bug}).

\begin{theorem}\label{teo:Mat}
Let $\K$ be a number field of degree $D$ over $\Q,\,\,$  $\G_1, \ldots, \G_t$ be positive
real numbers of $\K$, and $b_1, \ldots,  b_t$ rational integers. Put
$$
\Lambda := \G_1^{b_1} \cdots \G_t^{b_t}-1
\qquad
\text{and}
\qquad
B \geq \max\{|b_1|, \ldots ,|b_t|\}.
$$
Let $A_i \geq \max\{Dh(\G_i), |\log \G_i|, 0.16\}$ be real numbers, for
$i = 1, \ldots, t.$
Then, assuming that $\Lambda \not = 0$, we have
\[
|\Lambda| > \exp(-1.4 \times 30^{t+3} \times t^{4.5} \times D^2(1 + \log D)(1 + \log B)A_1 \cdots A_t).
\]
\end{theorem}
In the above and in what follows, for an algebraic number $\eta$ of degree $d$ over $\mathbb{Q}$ and minimal primitive polynomial over the integers
\[
f(X):=a_0\prod_{i=1}^{d}(X-\eta^{(i)}) \in \mathbb{Z}[X],
\]
with positive leading coefficient $a_0$, we write $h(\eta)$ for its logarithmic height, given by
\[
h(\eta):=\frac{1}{d}\left(\log a_0+\sum_{i=1}^{d}\log\left(\max\{|\eta^{(i)}|,1\}\right)\right).
\]
In particular, if $\eta=p/q$ is a rational number with $\gcd(p,q)=1$ and $q>0$, then $h(\eta)=\log \max \{|p|,q\}$. The following properties of the function logarithmic height $h(\cdot)$, which will be used in the next sections without special reference, are also known:
\begin{align*}
h(\eta\pm\gamma) & \leq h(\eta) + h(\gamma)+\log 2,\\
h(\eta\gamma^{\pm 1}) & \leq  h(\eta)+h(\gamma),\\
h(\eta^{s}) & =  |s|h(\eta)\qquad (s\in \Z).
\end{align*}
We now apply the Theorem \ref{teo:Mat} with  $t:=3$ and the parameters given in \eqref{param-1}. We begin by noting that the algebraic number field containing $\gamma_1,\gamma_2,\gamma_3$ is $\mathbb{K}:=\mathbb{Q}(\alpha)$, so we can take $D=[\mathbb{K}: \mathbb{Q}]:=3$. Since $h(\gamma_1)=(\log\A)/3$ and $h(\gamma_2)=\log b$, we can take $A_1:=\log\A$ and $A_2:= 3\log b$. We need to estimate $h(\gamma_3)$. First of all, by the properties of the logarithmic height, we obtain that
\begin{equation}\label{h(C_A)}
h(\gamma_3) \le h\left(\frac{a}{b-1}\right) + h(C_{\A})=\log(b-1) +h(C_{\A}).
\end{equation}
On the other hand, since $31x^3-31x^2+10x-1$ is the minimal primitive polynomial of $C_{\A}$ over $\mathbb{Z}$ and taking into account that all the zeros of it are strictly inside the unit circle, we get that $h(C_{\A})=(\log 31)/3$. Hence, it follows from \eqref{h(C_A)} that 
\[
h(\gamma_3) \le \log(b-1) +\frac{\log 31}{3}<2\log b.
\]
Thus we can take $A_3:=6\log b$. Furthermore,  since $\max\{n+2,\ell,1\}=n+2$ by Lemma \ref{bound_l}, we can take $B:=n+2$. To apply Theorem \ref{teo:Mat}, we need to prove that $\Lambda_1 \neq 0$. Observe that imposing that $\Lambda_1=0$, we get
\begin{equation}{\label{Gamma}}
C_\A\A^{n+2} = \frac{ab^{\ell}}{b-1}.
\end{equation}
Let $G$ be the Galois group of the splitting field of $f(x)$ over $\mathbb{Q}$, and let $\sigma\in G$ be an automorphism such that $\sigma(\A) =\B$. The action of $\sigma$ on the above relation \eqref{Gamma} concludes that
\[
|C_\B \B^{n+2}| = \frac{ab^\el}{b-1}.
\]
The above equality is not possible since $|C_\B \B^{n+2}|<|C_\B|=0.407506\ldots<1$, whereas $ab^\el/(b-1)\geq 4$ for $\ell\geq 2$. This is a contradiction. Thus $\L_1 \neq 0$.

Therefore using Theorem \ref{teo:Mat} we get the following lower bound for $|\L_1|$:
\begin{eqnarray*}
\exp(-1.4 \times 30^{6} \times 3^{4.5} \times  3^2(1 + \log 3)(1 + \log(n+2))(\log \A)(3\log b)(6\log b)),
\end{eqnarray*}
which is smaller than $6/\A^{n-m}$ by inequality \eqref{desl_1}. We take logarithm on both sides to conclude that
\[
(n-m)\log \alpha-\log 6 < 2 \times 10^{13}\,(1 + \log(n+2))\,{\log^{2} b}.
\]
Since $1 + \log(n+2)\leq 2\log n$ for $n\ge 5$, we obtain
\begin{equation}
\label{bound1}
(n-m)\log \A<10^{14}\log n \log^{2} b.
\end{equation}
In order to find an upper  bound on $n$ in terms of $b$, we  return to our equation \eqref{eqn} and rewrite it as
\[
C_\A \A^{n+2} + C_\A \A^{m+2} - \frac{ab^\el}{b-1}= -\z_n - \z_m - \frac{a}{b-1}.
\]
This implies that
\[
\left|C_\A \A^{n+2} (1+ \A^{m-n}) - \frac{ab^\el}{b-1}\right| < 2.
\]
Now dividing by $C_\A \A^{n+2} (1+ \A^{m-n})$ we obtain 
\begin{equation}\label{linform2}
\left|\A^{-(n+2)} b^{\el} \frac{a}{(b-1)C_\A (1+\A^{m-n})} - 1\right| < \frac{2}{C_\A \A^{n+2} (1+ \A^{m-n})} < \frac{5}{\A^n}.
\end{equation}
In a second application of Theorem \ref{teo:Mat}, we take the parameters $t:=3$ and
\[
\begin{split}
~~~~\gamma_1 &:= \A, \qquad \gamma_2:= b,  ~~~~\qquad \gamma_3:=\frac{a}{(b-1)C_\A (1+\A^{m-n})},\\
& b_1:= -(n+2), \qquad b_2:= \el,\qquad b_3:=1,~~~~~~\\
& \quad\qquad\qquad\Lambda_2 := \gamma_1^{b_1}\cdot\gamma_2^{b_2}\cdot \gamma_3^{b_3} - 1.~~
\end{split}
\]
We have from equation \eqref{linform2},
\begin{equation}\label{desl_2}
|\Lambda_2|< \frac{5}{\A^n}.
\end{equation}
We use similar arguments as before to conclude that $\L_2\neq 0$. In this application, we take $\K:=\Q(\A)$, $D:=3$, $A_1:=\log\A$, $A_2:= 3\log b$ and $B:=n+2$ as we did before. We begin with the observation that
\begin{align*}
h(\gamma_3) &\le h\left(\frac{a}{(b-1)C_{\A}}\right) + h(1+\A^{m-n})\\
          & \leq  2\log b + h(\A^{m-n}) + \log 2 \\
          & =  2\log b + |m-n|\, h(\A) + \log 2 \\
          & \leq  3\log b +  \frac{(n-m)\log \A}{3}.
\end{align*}
Hence from \eqref{bound1}, we get 
\[
h(\gamma_3)< 3\log b + \frac{10^{14}\log n \log^{2} b}{3}.
\]
Therefore we can take $A_3:= 1.1 \times 10^{14}\log n\log^{2} b$. This will allow us to obtain a lower bound for $|\L_2|$. Then we compare the lower bound for $|\L_2|$ from Theorem \ref{teo:Mat} with the upper bound of $|\L_2|$ from inequality \eqref{desl_2} to conclude that
\[
n\log \alpha-\log 5 < 7.3 \times 10^{26} \log^{2}n \log^{3} b.
\]
Thus
\[
n< 2 \times 10^{27}\log^{2} n\log^{3} b,
\]
which can be written as
\begin{equation}\label{bound2}
\frac{n}{{\log^{2} n}} < 2 \times 10^{27}\log^{3} b.
\end{equation}
We next quote an analytical argument that leads to an upper bound of $n$ in terms of $b$. The following result was proved by Guzm\'an and Luca \cite[Lemma 7]{GL}.

\begin{lemma}
\label{Lem:boundxlogmx}
If $x$ and $T$ are real numbers such that $T>16^2$ and
\[
\frac{x}{{\log^2 x}}< T, \quad \text{then} \quad x<4T \log^2 T.
\]
\end{lemma}
Taking $T:=2 \times 10^{27}\log^{3} b$, and applying the above Lemma in inequality \eqref{bound2}, we obtain
\begin{align*}
n &< 4 \, (2 \times 10^{27} \log^{3} b) \left(\log\left( 2 \times 10^{27}\,{\log^{3} b}\right)\right)^2\\
&< (8\times 10^{27} \log^{3} b)(63 +3\log\log b)^2\\
&< (8\times 10^{27} \log^{3} b) (90 \log b)^2\\
& < 6.5 \times 10^{31} \log^{5} b.
\end{align*}
In the above inequality, we have used the fact that $63+3\log\log b < 90\log b$ which is true for all $b\geq 2$.

We know from Lemma $\ref{bound_l}$ and equation \eqref{eqn} that $\el < n$ and $m\leq n$, respectively. At this point, we summarize the result we obtained so far on the upper bound of $n$. The result is the following: 
\begin{theorem}\label{ReductionThm}
Let $(n,m,\el,a,b)$ be a solution of equation \eqref{eqn} with $\ell \geq 2$, $b\ge 2$ and  $1\leq a \leq b-1$, then
\[
\max \{\el, m\} \leq n< 6.5 \times 10^{31} \log^{5} b.
\]
\end{theorem}
\begin{remark}
For a fixed base $b$ $(\geq2)$ the equation \eqref{eqn} has only finitely many solutions.
\end{remark}
In the next section, we reduce the bounds on $n$ using a reduction method due to Dujella and Peth\H{o} \cite{DP}, which is a generalization of a classical result of Baker and Davenport \cite{BD}.

\section{Reduction lemma and the reduced bounds}
We begin this section with the following simple facts of the exponential function. We list it as a lemma for further reference.

\begin{lemma}\label{exponential}
For any non-zero real number $x$, we have 
\begin{enumerate}
    \item[$(a)$] $0< x < |e^x - 1|$.
    \item[$(b)$] If $x<0$ and $|e^x - 1| <1/2$, then $|x| < 2\, |e^x - 1|$.
\end{enumerate}
\end{lemma}
We write 
\[
z_1= \el\log b - (n+2)\log \A +\log{\left(\frac{a}{(b-1) C_{\A}}\right)}.
\]
Notice that $z_1 \ne 0$ as $e^{z_1} -1 = \L_1 \ne 0$. 
\begin{lemma}\label{case1}
Let $m=0$ and suppose that $n\geq 7$. Then
\[
0< |z_1|< \frac{12}{\A^{n}}.
\]
\end{lemma}

\begin{proof}
Since $m=0$, inequality \eqref{First_Lin_Form} can be written in the form $| e^{z_1} -1 | < 6/\A^{n}$. If $z_1>0$, then we can apply Lemma \ref{exponential} $(a)$ to obtain $|z_1|=z_1<|e^{z_1} -1|<6/\A^{n}$. On the contrary, if $z_2<0$, then $| e^{z_1} -1 | <6/\A^{n} <1/2$ for all $n\ge 7$.  It then follows from Lemma \ref{exponential} $(b)$ that $|z_1|< 2| e^{z_1} -1 | <12/\A^{n}$. In both cases, we get that $|z_1|<12/\A^{n}$ which holds for all $n\geq 7$.
\end{proof}

\begin{lemma}\label{case2}
Let $m\geq 1$. Then
\[
0< z_1< \frac{6}{\A^{n-m}}.
\] 
\end{lemma}

\begin{proof}
First of all we note that \eqref{First_Lin_Form} can be rewritten as
\[
| e^{z_1} -1 | < \frac{6}{\A^{n-m}}.
\]
Furthermore, using equation \eqref{eqn} and Lemma \ref{lem:Prop}, we have
\[
C_\A \A^{n+2}= N_n - \z_n < N_n + \frac{1}{2} < N_n + N_m=a \left(\frac{b^\ell-1}{b-1}\right)<\frac{ab^\ell}{b-1},
\]
and so $z_1>0$. Hence, in view of  Lemma \ref{exponential} $(a)$, we conclude that $z_1 < | e^{z_1} -1 | <6/\A^{n-m}$.
\end{proof}

We note that the upper bound obtained in Theorem \ref{ReductionThm} is very large and depends on the base $b$.  We also note that we did not put any restrictions on the base $b$ so far. Next we restrict $b$ in the set $\{2,\ldots, 100\}$ that will enable us to do the computation. We note that the same computation can be done possibly for a larger set of values for $b$ but it will not add anything significantly new to the result. So we stop at $b\leq 100.$

The following lemma by Bravo, G\'omez, and Luca \cite{BGL-prep-14} is a slight variation of a result due to Dujella and Peth\H{o} \cite{DP}, which itself is a generalization of a result of Baker and Davenport \cite{BD}.  We will use this lemma for the reduction of the bounds on $n$. 

\begin{lemma}\label{reduce}
Let $A,B,\widehat{\gamma},\widehat{\mu}$ be positive real numbers and $M$ a positive integer. Suppose that $p/q$ is a convergent of the continued fraction expansion of the irrational $\widehat{\gamma}$ such that $q>6M$. Put $\epsilon:=||\widehat{\mu} q||-M||\widehat{\gamma} q||$, where $||\cdot||$ denotes the distance from the nearest integer. If $\epsilon >0$, then there is no positive integer solution $(u,v,w)$ to the inequality
\begin{equation*} \label{expDP}
0<|u\widehat{\gamma}-v+\widehat{\mu}|<AB^{-w},
\end{equation*}
subject to the restrictions that
\[
u\leq M \quad\text{and}\quad w\geq \frac{\log(Aq/\epsilon)}{\log B}.
\]
\end{lemma}

\subsection{Reduction of $n$: Step 1}

If $m=0$, then from Lemma \ref{case1} we have
\begin{equation}\label{RedEqn2}
 0<\left| \el \left(\frac{\log b}{\log \A}\right) -n + \left(\frac{\log\left(a/((b-1)C_\A)\right)}{\log \A} -2 \right)\right| < 32\A^{-n}.
\end{equation}
We are now in a position to apply Lemma \ref{reduce} with the parameters $u := \el$, $v := n$, $w := n$, $\widehat{\gamma} = (\log b)/\log\A$, $A := 32$, $B := \A$ and 
\[
\widehat{\mu}:= \frac{\log\left(a/((b-1)C_\A)\right)}{\log \A} -2.
\]
It is clear that $\widehat{\G}$ is an irrational number because $\alpha>1$ is a unit in ${\mathcal O}_{\mathbb K}$, the ring of integers of $\mathbb{K}$. So $\alpha$ and $b$ are multiplicatively  independent. 

We take $M:= M_b = 6.5\times 10^{31}\log^5 b$. Then we apply Lemma \ref{reduce} on the inequality \eqref{RedEqn2} for all the choices of $b\in\{ 2, \ldots, 100\}$ and $a\in \{1, \ldots, b-1 \}$. For $m=0$, a simple computation in \emph{Mathematica} for all possible choices of $b$ allows us to conclude that a possible solution $(n,0,\ell,a,b)$ of the equation \eqref{eqn} satisfies $n\leq 260$. 

We now suppose that $m\geq 1$. In this case, from Lemma \ref{case2} we obtain
\begin{equation}\label{RedEqn1}
0< \el \left(\frac{\log b}{\log \A}\right) -n + \left(\frac{\log\left(a/((b-1)C_\A)\right)}{\log \A} -2 \right) < 16\A^{-(n-m)}.
\end{equation}
Again we apply Lemma \ref{reduce} to inequality \eqref{RedEqn1} for all the choices of $b \in \{ 2, \ldots, 100\}$ and $a\in \{1, \ldots, b-1 \}$. We conclude that the possible solutions $(n,m,\ell,a,b)$ of the equation \eqref{eqn} for which $m \geq 1$ satisfy $n-m\in[0,260]$.  

\subsection{Reduction on $n$: Step 2}

In this section, we use the previous bound on $n-m$ to obtain a suitable upper bound on $n$. In order to do this, we let 
\[
z_2= \el\log b- (n+2)\log \A + \log{\left(a/((b-1) C_{\A} (1 + \A^{(m-n)}))\right)},
\]
and we observe that \eqref{linform2} can be rewritten as
\begin{equation}\label{2da-red}
| e^{z_2} -1 | < \frac{5}{\A^{n}}.
\end{equation}
Notice that $z_2 \ne 0$ as $e^{z_2} -1 = \L_2 \ne 0$. Now we proceed as in the previous section to obtain, from inequality \eqref{2da-red} and Lemma \ref{exponential}, that
\[
0<|z_2|<\frac{10}{\A^{n}}.
\]
Replacing $z_2$ in the above inequality by its formula and dividing it across by $\log\alpha$, we conclude that
\begin{equation}\label{redfinal}
0<\left|\el \left(\frac{\log b}{\log \A}\right) -n + \left(\frac{\log{\left(a/((b-1) C_{\A} (1 + \A^{(m-n)}))\right)}}{\log \A} -2 \right)\right|<27\A^{-n}.
\end{equation}
We apply Lemma \ref{reduce} once again with the data $u := \el$, $v := n$, $w := n$, $\widehat{\G}: =(\log b)/\log\A$, $A := 27$, $B := \A$ and 
\[
\widehat{\mu}:=\frac{\log{\left(a/((b-1) C_{\A} (1 + \A^{(m-n)}))\right)}}{\log \A} -2.
\]
By taking $M:= M_b = 6.5 \times 10^{31} \log^5 b$, we  apply Lemma \ref{reduce} on inequality \eqref{redfinal} for all the  choices of $b \in \{ 2, \ldots, 100\}$, $a\in \{1, \ldots, b-1 \}$ and $n-m\in \{0,\ldots,260\}$. A computer search with \emph{Mathematica} finds that the possible solutions $(n,m,\ell,a,b)$ of the equation \eqref{eqn} all have $n\leq 280$.

\section{Proof of Theorem \ref{mainthm}}\label{sec5}
From the previous section, we conclude that the search for solutions $(n,m,\ell,a,b)$ to the Diophantine equation \eqref{eqn} with $0\leq m\leq n$, $2\le b\leq100$, $1\leq a \leq b-1$ and  $\ell\geq 2$ reduces to the range $1\le n \le 280$. We compute all the solutions with the help of \emph{Mathematica} for the above range. We note down all the solutions with $\ell \ge 4$ of equation \eqref{eqn}.
\begin{align*}
N_8 + N_7&=9+6=15=\frac{2^4 -1}{2-1}= \left[1,1,1,1\right]_{2} \\
N_9 + N_4 &=13+2=15=\frac{2^4 -1}{2-1}= \left[1,1,1,1\right]_{2} \\
N_{11} + N_5 &=28+3=31=\frac{2^5 -1}{2-1}= \left[1,1,1,1,1\right]_{2}  \\
N_{13} + N_5 &=60+3=63=\frac{2^6 -1}{2-1}= \left[1,1,1,1,1,1\right]_{2} \\
N_{15} + N_{12} &=129+41=170= 2 \left( \frac{4^4 -1}{4-1}\right)= \left[2,2,2,2\right]_{4} \\
N_{21} + N_{17} &=1278+277=1555=\frac{6^5 -1}{6-1}=\left[1,1,1,1,1\right]_{6}.
\end{align*}
In the following table we list down all the solutions $(n,m,\el, a, b)$ of the equation \eqref{eqn} with $\ell\geq 3$.

\begin{table}[ht]
\begin{center}
\begin{tabular}{| c | c | c | c | c |}
\hline
(6,5,3,1,2) & (7,1,3,1,2) & (7,2,3,1,2) & (7,3,3,1,2) & (8,6,3,1,3) \\
\hline
(8,7,4,1,2) & (9,0,3,1,3) & (9,4,4,1,2) & (9,9,3,2,3) & (10,4,3,1,4) \\
\hline
(11, 5, 5,1,2) & (11,5,3,1,5) & (12,1,3,2,4) & (12,2,3,2,4) & (12,3,3,2,4) \\
\hline
(12,4,3,1,6) & (13,4,3,2,5) & (13,5,6,1,2) & (13,5,3,3,4) & (13,9,3,1,8) \\
\hline
(14,5,3,1,9) & (14,12,3,3,6) & (15,0,3,3,6) & (15,6,3,1,11) & (15,11,3,1,12) \\
\hline
(15,12,4,2,4) & (17,14,3,5,8) & (19,7,3,1,24) & (19,10,3,2,17) & (21,5,3,7,13) \\
\hline
(21,15,3,1,37) &  (21,17,5,1,6) & (21,18,3,4,20) & (26,20,3,9,32) & (26,22,3,2,72) \\
\hline
(28,13,3,20,30) & (30,18,3,11,60) &  &  & \\   
\hline
\end{tabular}
\caption{Solutions of equation \eqref{eqn} with $\ell\geq 3$}\label{table1}
\end{center}
\end{table}

\subsection{Consequences} \label{pag}

We can take $N_m=N_0=0$ in equation \eqref{eqn} and hence the Corollary \ref{coro} follows immediately. The Corollary \ref{coro2} holds with the choices of bases $10$ and $2$, respectively. There are some other interesting corollaries of Theorem \ref{mainthm}.

\begin{definition}
We define an $(m\times 1)$ array of decimal integers as an $m$-block.
\end{definition}
\begin{example}
\fbox{26}, \fbox{582}, \fbox{29156} represent a $2$-, $3$- and $5$- blocks, respectively. 
\end{example}
\begin{definition}
Let $N$ be a positive integer and suppose $N$ can be written as $n$ repetitive $m$-blocks. Then we call it an $m$-block repdigit of length $n$.
\end{definition}
For example,
\[
N=\underbrace{\fbox{ab}\cdots\fbox{ab}}_\text{$n$ repetitive\ 2-blocks}
\]
where $a,b \in \{0,\cdots,9\}$ and $(a,b)\neq (0,0)$ denotes a $2$-block repdigit of length $n$.

Alternatively, $N$ is a repdigit in base $100$. We can study these special repdigits with bases $b=10^u$ in recurrence sequences for postive integers $u$. As a consequence of Theorem \ref{mainthm}, we obtain the following corollaries.

\begin{corollary}
There are no $1$-block repdigits of length $\geq 3$ in the Narayana sequence. In fact, the only $1$-block repdigit of length $2$ is given by $N_{14}= \fbox{8}\fbox{8}$.
\end{corollary}

\begin{corollary}
There are no $2$-block repdigits of length $\geq 2$ in the Narayana sequence. In other words, $N_n= a(100^{\el} -1)/(100-1)$ has no solutions with $\el \ge 2$ and $a\in\{1,\ldots,99\}.$
\end{corollary}
We believe $m$-block repdigits for arbitrary integer $m$ are extremely rare in  the Narayana sequence. After a numerical evidence we pose the following conjecture in this context.

\begin{conjecture}
Let $m\geq2$ be an arbitrary integer. There are no $m$-block repdigits of length $\geq 2$ in the Narayana sequence.
\end{conjecture}

\section{Acknowledgment}
J.~J.~B. was partially supported by Projects VRI ID~4689 (Universidad del Cauca) and Colciencias~110371250560.

\end{document}